\documentclass[11pt]{amsart}
\usepackage{fullpage}
\usepackage{amsfonts,amscd}
\usepackage{amssymb}
\usepackage{url}
\usepackage{graphicx}
\usepackage[english]{babel}
\usepackage{listings}

\usepackage{mathtools}

\usepackage[margin=2cm]{geometry}
\usepackage[utf8]{inputenc}
\usepackage[T1]{fontenc}
\usepackage{lmodern}
\usepackage{textcomp}

\theoremstyle{plain}
\newtheorem{theorem}                {Theorem}      [section]
\newtheorem{proposition}  [theorem]  {Proposition}
\newtheorem{corollary}    [theorem]  {Corollary}
\newtheorem{lemma}        [theorem]  {Lemma}

\theoremstyle{definition}

\newtheorem{remark}       [theorem]  {Remark}

\numberwithin{equation}{section}

\def \R{{\mathbb R}}
\def \s{{\mathbb S}}

\usepackage{color}
\def \link{~}

\begin{document}

\title[]{On the stability of the equator map for higher order energy functionals}

\author{A.~Fardoun}
\address{Universit\'e de Bretagne Occidentale\\
D\'epartement de Math\'ematiques\\
6, Av. Victor Le Gorgeu\\
B.P. 809 \\
29285 Brest Cedex, France}
\email{ali.fardoun@univ-brest.fr}

\author{S.~Montaldo}
\address{Universit\`a degli Studi di Cagliari\\
Dipartimento di Matematica e Informatica\\
Via Ospedale 72\\
09124 Cagliari, Italia}
\email{montaldo@unica.it}

\author{A.~Ratto}
\address{Universit\`a degli Studi di Cagliari\\
Dipartimento di Matematica e Informatica\\
Via Ospedale 72\\
09124 Cagliari, Italia}
\email{rattoa@unica.it}

\begin{abstract} 
Let $B^n\subset \R^{n}$ and $\s^n\subset \R^{n+1}$ denote the Euclidean $n$-dimensional unit ball and sphere respectively. The \textit{extrinsic $k$-energy functional} is defined on the Sobolev space $W^{k,2}\left (B^n,\s^n \right )$ as follows: $E_{k}^{{\rm ext}}(u)=\int_{B^n}|\Delta^s u|^2\,dx$ when $k=2s$, and $E_{k}^{{\rm ext}}(u)=\int_{B^n}|\nabla \Delta^s u|^2\,dx$ when $k=2s+1$. These energy functionals are a natural higher order version of the classical extrinsic bienergy, also called Hessian energy. The equator map $u^*: B^n \to \s^n$, defined by $u^*(x)=(x/|x|,0)$, is a critical point of $E_{k}^{{\rm ext}}(u)$ provided that $n \geq 2k+1$. The main aim of this paper is to establish necessary and sufficient conditions on $k$ and $n$ under which $u^*: B^n \to \s^n$ is minimizing or unstable for the extrinsic $k$-energy. 
\end{abstract}



\thanks{Work supported by Fondazione di Sardegna (Project GESTA) and Regione Autonoma della Sardegna (Project KASBA)}

\maketitle

\section{Introduction and statement of the results}\label{intro}
Let $B^n$ and $\s^n$ denote the Euclidean $n$-dimensional unit ball and sphere respectively, and let $k$ be a positive integer. We shall work in the Sobolev spaces

\begin{equation*}\label{Sobolev-space}
W^{k,2}\left ( B^n,\s^n \right )=\left \{ u \in W^{k,2}\left ( B^n,\R^{n+1} \right )\,\,: \,\, u(x)=\left (u^1(x),\ldots,u^{n+1}(x)\right ) \in \s^n \,\,{\rm a.e.} \right \} \,\,.
\end{equation*}
The \textit{extrinsic $k$-energy functional} $E_{k}^{{\rm ext}}(u)$ is defined on $W^{k,2}\left (B^n,\s^n \right )$ as follows:
\begin{equation}\label{2s-energia-ext}
    E_{k}^{{\rm ext}}(u)=\int_{B^n}\,\left | \Delta^s u \right |^2\,dx \qquad \,\,\,\,\quad {\rm when}\;k=2s\,;
\end{equation} 
\begin{equation}\label{2s+1-energia-ext}
   \qquad E_{k}^{{\rm ext}}(u)=\int_{B^n}\,\left |\nabla \Delta^s u \right |^2\,dx \qquad \quad{\rm when}\; k=2s+1\,.
\end{equation}

We say that $u \in W^{k,2}\left (B^n, \s^n \right)$ is an \textit{extrinsic (weakly) $k$-harmonic map} if 

\[
\left .\frac{d}{dt} E_k^{\rm ext} (u_t)\right |_{t=0} = 0 
\]
for all variations 
\[
u_t = \Pi(u + t\phi )\,,
\] where $\phi \in  C_0^{\infty} \left(B^n, \R^{n+1}\right)$ and $\Pi$ is the nearest point projection onto $\s^n$. 

A very important class of critical points are the so-called minimizers. More specifically, a \textit{minimizer}, or \textit{minimizing extrinsic $k$-harmonic map}, is a map $u \in W^{k,2}\left (B^n,\s^n \right )$ such that
\begin{equation*}\label{def-minimizer}
    E_{k}^{{\rm ext}}(u)\leq   E_{k}^{{\rm ext}}(v)
\end{equation*}
for all $v \in W^{k,2}\left (B^n,\s^n \right )$ such that $u-v \in W^{k,2}_0\left(B^n,\R^{n+1}\right )$. 

The study of the regularity of minimizers for the energy and the bienergy has a long history which began in the late seventies with the celebrated works of Hildebrandt et al. \cite{MR433502} and Uhlenbeck \cite{MR474389}. As for the biharmonic case, we cite \cite{MR1692148} and \cite{MR2153032}, where it was shown that every minimizing biharmonic map from a domain $\Omega \subset \R^m$ to $\s^n \,\,(m \geq5$) is smooth away from a closed set whose Hausdorff dimension is at most $m-5$. Other significant contributions in this field can be found in \cite{MR1101039,MR2389639,MR2854719,MR2566701,MR762354,MR2413109,MR2054520}.

If $u$ is an extrinsic $k$-harmonic map, we say that $u$ is \textit{stable} if
\begin{equation*}\label{second-variation}
 \left . \frac{d^2}{dt^2}\, E_{k}^{{\rm ext}}\left (u_t \right ) \right.\Big|_{t=0} \, \geq \, 0 
\end{equation*}
for all variations $u_t$ such that $u_t-u \in W^{k,2}_0\left(B^n,\R^{n+1}\right )$ . 

Clearly, if $u$ is an unstable critical point, then it cannot be a minimizer.

In this work we shall focus on the so-called equator map $u^*: B^n \to \s^n$ which is defined by

\[
\begin{array}{rccc}
u^* : &B^n &\to   &\s^n \subset\R^n \times \R \\
 & x&\xmapsto{\phantom{{\quad\quad}}}&\left ( \frac{x}{|x|} \,,\,0\right )
 \end{array}
\]

As we shall prove in Section\link\ref{proofs}, the first, simple step in our analysis is the following proposition:
\begin{proposition}\label{prop-eq-map}
The equator map $u^*:B^n \to \s^n$ belongs to $W^{k,2}\left (B^n,\s^n \right )$ if and only if $n \geq 2k+1$. Moreover, if $n \geq 2k+1$, it is an extrinsic $k$-harmonic map.
\end{proposition}
The main aim of this paper is to establish necessary and sufficient conditions on $k$ and $n$ under which $u^*: B^n \to \s^n$ is minimizing or unstable for the extrinsic $k$-energy functional. 

First, in order to motivate our work and set it in an appropriate context, we briefly recall some basic facts about the well-known cases $k=1,2$. 

In the case of maps $u: B^n \to \s^n$, our extrinsic energy \eqref{2s+1-energia-ext} with $s=0$ coincides, up to an irrelevant constant factor $1/2$, with the classical {\em energy functional} whose  critical points are the usual {\it harmonic maps}. We refer to \cite{ MR703510,MR1363513} for background on harmonic maps. 

Writing, here and below, $u$ for $u \circ i$, $i:\s^n \hookrightarrow \R^{n+1}$, the harmonicity equation is
\begin{equation}\label{Euler-Lag-equation-1-energy-ext}
\Delta u + \lambda_1 u=0 
\end{equation}
with
\begin{equation*}\label{lambda-explicit-1-energy}
\lambda_1= - \langle u, \Delta u \rangle = | \nabla u |^2 \,.
\end{equation*}
Note that $\langle\,,\, \rangle$ denotes the Euclidean scalar product 
and 
\begin{equation*}\label{gradiente-Laplaciano}
\nabla u= \left (\nabla u^1,\ldots, \nabla u^{n+1} \right )\,,\quad 
\quad \Delta u= \left (\Delta u^1,\ldots, \Delta u^{n+1} \right )\,\,.
\end{equation*}
Each entry of $\nabla u$ is an $n$-dimensional vector and our sign convention for the Laplace operator $\Delta$ is such that $\Delta f = f''$ on $\R$. 

J\"{a}ger and Kaul (see \cite{MR705882}) carried out a detailed qualitative study of rotationally symmetric solutions of \eqref{Euler-Lag-equation-1-energy-ext} and proved the following result concerning the equator map:
\begin{theorem}\label{Th-Jager-Kaul} The equator map $u^*: B^n \to \s^n$ is
\begin{itemize}
\item[{\rm (i)}] a minimizing harmonic map if $n \geq 7$;
\item[{\rm (ii)}] an unstable harmonic map if $3 \leq n <7$.
\end{itemize} 
\end{theorem}

When $k\geq2$ the difference between the intrinsic case and the extrinsic one appears.

The intrinsic approach, as suggested in \cite{MR703510}, \cite{MR0216519},  is the study of the so-called \textit{bienergy functional} whose critical points are called {\it biharmonic maps}. 
There have been extensive studies on biharmonic maps and we refer to \cite{MR3362186, MR4076824,MR886529, MR2375314,MR2301373,MR3357596} for an introduction to this topic. 

In the case of maps $u:B^n \to \s^n$ the bienergy functional, up to an irrelevant constant factor $1/2$, takes the following form:
\begin{equation}\label{bienergia-target-sphere}
    E_2(u)=\int_{B^n}\,|(\Delta u)^T|^2\,dx=\int_{B^n}\,\left ( |\Delta u|^2 - |\nabla u|^4\right )\,dx\,,
\end{equation} 
where $( \,\, )^T$ denotes the tangential component. 

In general, it is very difficult to apply variational methods and, particularly, direct minimization to deduce the existence of \textit{proper} (i.e., not harmonic) biharmonic maps. The main reason for this is the fact that harmonic maps trivially provide absolute minima for the bienergy. 

To overcome this difficulty, an interesting variant of \eqref{bienergia-target-sphere} has been introduced. This is precisely the functional $ E_{2}^{{\rm ext}}(u)$ defined in \eqref{2s-energia-ext} with $s=1$. This functional, called \textit{extrinsic bienergy} or \textit{Hessian energy}, depends on the embedding of $\s^n$ into the Euclidean space and a harmonic map into $\s^n$ is not necessarily a minimizer for $ E_{2}^{{\rm ext}}(u)$ (actually, it may even not be a critical point).

On the other hand, the study of the extrinsic bienergy is very rich from the point of view of the stability and regularity of solutions to a nonlinear elliptic system in a geometrically significant setting. Important contributions in this area can be found in \cite{MR2332421, MR1692148, MR2845758, MR2322746, MR2153032, MR2398228, MR2450176,  MR2054520}.

The Euler-Lagrange equation of $ E_{2}^{{\rm ext}}(u)$ is
\begin{equation*}\label{Euler-Lag-equation-2-energy-ext}
\Delta^2 u + \lambda_2 u=0 
\end{equation*}
with
\begin{equation*}\label{lambda-explicit-2-energy}
\lambda_2=\Delta\left ( \left |\nabla u \right |^2 \right )+\left |\Delta u \right |^2+2\,\left ( \nabla u .  \nabla \Delta u \right )  \,\,,
\end{equation*}
where, here and below, ``$.$'' denotes scalar product in the following sense:
\begin{equation}\label{def-scalar-prod}
\nabla u .  \nabla \Delta u= \sum_{j=1}^{n+1}\,\langle \nabla u^j ,  \nabla \Delta u^j \rangle\,.
\end{equation}
Hong and Thompson proved the following very interesting result:
\begin{theorem}\label{Th-equator-map-Hong} \cite{MR2322746} The equator map $u^*: B^n \to \s^n$ is
\begin{itemize}
\item[{\rm (i)}] a minimizing extrinsic biharmonic map if $n \geq 10$;
\item[{\rm (ii)}] an unstable extrinsic biharmonic map if $5 \leq n <10$.
\end{itemize} 
\end{theorem} 
When $k \geq3$, an intrinsic approach was proposed in \cite{MR703510,MR0216519} and several basic results were obtained in a series of paper by Maeta et al. (see \cite{MR2869168, MR3403738,MR3314128, MR3371364}). For recent progresses in this context and an updated bibliography we refer to \cite{MR4106647} and references therein. 

On the other hand, as we already pointed out for the case $k=2$, the extrinsic energy functionals \eqref{2s-energia-ext}, \eqref{2s+1-energia-ext} are more suitable to apply methods of elliptic analysis and calculus of variations.

More specifically, the main aim of this paper is to prove a version of Theorems\link\ref{Th-Jager-Kaul} and \ref{Th-equator-map-Hong} in the case that $k \geq3$. 

In order to state our result, it is convenient to introduce the following two sets of constants which depend on $n$ and $k$. First, we set
\begin{equation}\label{def-A-j-n}
A_{k}(n)=(-1)^k \, \, \prod_{i=1}^k(n-2i+1)(2i-1)\,.
\end{equation}
Next, we define $\alpha_k(n)$ as follows:

\begin{equation}\label{def-alpha-k}
\begin{array}{lcll}
\alpha_{2s}(n)&=&\displaystyle{\frac{1}{2^{4s}} \, \prod_{i=1}^s\Big [ (n - 4 i )^2 (n + 4 i -4)^2 \Big ]} & (s \geq1) \\
\alpha_{2s+1}(n)&=&\displaystyle{\frac{1}{2^{4s+2}} \,(n - 
    2)^2  \prod_{i=1}^s\Big [ (n - 4 i  -2)^2 (n + 4 i -2)^2 \Big ]}& (s \geq 1) \\
    \alpha_1(n)&=&\dfrac{(n-2)^2}{4}\,.& 
    \end{array}
\end{equation}

As we shall see in Section\link\ref{proofs}, the constants $A_k(n)$ arise in the computation of $\Delta^k u^*$, while the $\alpha_k(n)$'s occur in a family of high order Hardy inequalities. 

Finally, in the notation of \eqref{def-A-j-n}, \eqref{def-alpha-k}, we define the following family of polynomials $\mathcal{P}_k(n)$, $k \geq 1$:
\begin{equation}\label{def-P-k-n}
\begin{array}{lcll}
\mathcal{P}_{2s}(n)&=&\alpha_{2s}(n)-A_{2s}(n) &(s \geq1) \\
\mathcal{P}_{2s+1}(n)&=&\alpha_{2s+1}(n)+A_{2s+1}(n) &(s \geq 0)\,. 
\end{array}
\end{equation}
Our main result is
\begin{theorem}\label{Th-equator-map-k>4} Let $\mathcal{P}_k(n)$ be the polynomial defined in \eqref{def-P-k-n} and assume $n \geq 2k+1$. Then the equator map $u^*:B^n \to \s^n$ is
\begin{itemize}
\item[{\rm (i)}] an energy minimizing extrinsic $k$-harmonic map if $\mathcal{P}_{k}(n)\geq0$;
\item[{\rm (ii)}] an unstable extrinsic $k$-harmonic map if $\mathcal{P}_{k}(n)<0$.
\end{itemize}
\end{theorem}
In the case that $k=1$ the extrinsic energy coincides with the classical intrinsic energy of harmonic maps and the analysis of 
\[
\mathcal{P}_1(n)=\frac{1}{4} \left(n^2-8 n+8\right)
\] 
shows immediately that $\mathcal{P}_{1}(n)\geq0$ if and only if $n \geq7$, and so we recover Theorem\link\ref{Th-Jager-Kaul}. Similarly, 
\[
\mathcal{P}_2(n)=\frac{1}{16} (n-4)^2 n^2-3 (n-3) (n-1)
\]
and it is easy to check that $\mathcal{P}_{2}(n)<0$ when $5 \leq n \leq 9$, while $\mathcal{P}_{2}(n)\geq0$ if $n \geq 10$: this is the result of Hong and Thompson stated in Theorem\link\ref{Th-equator-map-Hong}. 

Next, it is important to discuss more in detail the consequences and specific applications of our Theorem\link\ref{Th-equator-map-k>4}. To give a first idea of these applications, the analysis of $\mathcal{P}_{3}(n)$ and $\mathcal{P}_{4}(n)$, as described in Remark\link\ref{Rem-n*k}, leads us to the following corollaries:
\begin{corollary}\label{Cor-equator-map-k=3}
The equator map $u^*: B^n \to \s^n$ is
\begin{itemize}
\item[{\rm (i)}] a minimizing extrinsic triharmonic map if $n \geq 12$;
\item[{\rm (ii)}] an unstable extrinsic triharmonic map if $7 \leq n <12$.
\end{itemize} 
\end{corollary}
\begin{corollary}\label{Cor-equator-map-k=4}
The equator map $u^*: B^n \to \s^n$ is
\begin{itemize}
\item[{\rm (i)}] a minimizing extrinsic $4$-harmonic map if $n \geq 15$;
\item[{\rm (ii)}] an unstable extrinsic $4$-harmonic map if $9 \leq n <15$.
\end{itemize} 
\end{corollary}
Next, we illustrate in more detail the consequences of Theorem\link\ref{Th-equator-map-k>4} when the integer $k$ is arbitrary. 

The following technical lemma, which we shall prove in Section\link\ref{proofs}, contains some information which is very useful for our purposes.
\begin{lemma}\label{lemma-Pj(n)}
\begin{equation}\label{Pj(2j+1)<0}
\mathcal{P}_k(2k+1) <0 \quad \forall k \geq 1 \,.
\end{equation}
Moreover, for all $k \geq 1$, there exists an integer $n_k^* > 2k+1$ such that the $\mathcal{P}_k(n) \geq 0$ if $n \geq n_k^*$, and $\mathcal{P}_k(n) < 0$ if $2k+1 \leq n < n_k^*$.
\end{lemma}
Now, as an immediate consequence of Theorem\link\ref{Th-equator-map-k>4} and Lemma\link\ref{lemma-Pj(n)}, we can state:
\begin{corollary}\label{Cor-equator-map-k>4} Let $n_k^*$ be the integer of Lemma\link\ref{lemma-Pj(n)}. Then the equator map $u^*: B^n \to \s^n$ is
\begin{itemize}
\item[{\rm (i)}] a minimizing extrinsic $k$-harmonic map if $n \geq n_k^*$;
\item[{\rm (ii)}] an unstable extrinsic $k$-harmonic map if $2k+1 \leq n <n_k^*$.
\end{itemize} 
\end{corollary}
\begin{remark} \label{Rem-n*k}
The degree of the polynomials $\mathcal{P}_k(n)$ is $2k$. Therefore, finding a general formula to compute $n_k^*$ seems to be a very difficult task because one has to compare the constants $\alpha_k(n)$ with that of type $A_{k}(n)$, and the $\alpha_k(n)$'s factor into the product of terms of the type $(n\pm 2j)$, while the $A_{k}(n)$'s are built up with terms of the form $(n-2i+1)$ or $(2i-1)$. 

Nevertheless, using the method of the proof of \eqref{Pj(2j+1)<0}, it is easy to check that $n_k^*< 4(k+1)$ for all $k\geq1$. This upper bound for $n_k^*$ is not sharp at all, but it enables us to write a simple computer algorithm which yields the value of $n_k^*$ when $k$ varies in a prescribed range (see Appendix\link\ref{Appendix}).

For instance, the following table, which includes Theorems\link\ref{Th-Jager-Kaul}, \ref{Th-equator-map-Hong} as well as Corollaries\link\ref{Cor-equator-map-k=3}, \ref{Cor-equator-map-k=4} as special cases, reports the values of $n_k^*$ for $1 \leq k \leq 40$:
\begin{equation}\label{table-n-j}
\begin{array}{lllll}
n_1^*=7 &n_2^*=10 & n_3^*= 12&n_4^*= 15&n_5^*=17  \\
n_6^*=19&n_7^*= 21&n_8^*=24 &n_9^*=26 &n_{10}^*= 28\\
n_{11}^*=30 &n_{12}^*=32 & n_{13}^*= 34&n_{14}^*= 36&n_{15}^*=39  \\
n_{16}^*=41 &n_{17}^*=43 & n_{18}^*= 45&n_{19}^*= 47&n_{20}^*=49  \\
n_{21}^*=51 &n_{22}^*=53 & n_{23}^*= 55&n_{24}^*= 57&n_{25}^*=59  \\
n_{26}^*=62&n_{27}^*=64 & n_{28}^*= 66&n_{29}^*= 68&n_{30}^*=70  \\
n_{31}^*=72 &n_{32}^*=74 & n_{33}^*= 76&n_{34}^*= 78&n_{35}^*=80  \\
n_{36}^*=82 &n_{37}^*=84 & n_{38}^*= 86&n_{39}^*= 88&n_{40}^*=90 \,.
\end{array}
\end{equation}
Inspection of \eqref{table-n-j} suggests that either $n_{k+1}^*= n_k^*+2$ (most of the occurrences) or $n_{k+1}^*= n_k^*+3$. However, the dynamic of this progression is still unclear. Of course, using our computer algorithm we can also determine $n_k^*$ for large values of $k$. For instance, we have
\begin{equation}\label{table-n-j-large}
\begin{array}{llll}

n_{2000}^*=4019 &n_{2001}^*=4021 & n_{2002}^*=4023&n_{2003}^*= 4025 \\
n_{2004}^*=4027 &n_{2005}^*=4029 & n_{2006}^*= 4031&n_{2007}^*= 4033 \\
&&& \\
n_{4000}^*=8021 &n_{4001}^*=8023 & n_{4002}^*= 8025&n_{4003}^*= 8027 \\
n_{4004}^*=8029 &n_{4005}^*=8031 & n_{4006}^*= 8033&n_{4007}^*= 8035 \\
&&&\\
n_{6000}^*=12022 &n_{6001}^*=12024 & & \\
n_{8000}^*=16023 &n_{8001}^*=16025 & & \\
n_{10000}^*=20023 &n_{10001}^*=20025 & & \\
n_{20000}^*=40025 &n_{20001}^*=40027\,. & & \\
\end{array}
\end{equation}
Inspection of \eqref{table-n-j-large} suggests that there is at least one instance of $n_{k+1}^*\geq n_k^*+3$ for $2000 \leq k \leq 4000$. Indeed, since  $n_{2000}^*=4019$ and $n_{4000}^*=8021$, it follows that $n_{k+1} \leq n_k +2$ for all $2000 \leq k \leq 4000$ would give a contradiction. Similarly, we have at least one instance of $n_{k+1}^*\geq n_k^*+3$ for both $4000 \leq k \leq 6000$ and $6000 \leq k \leq 8000$; possibly none for $8000 \leq k \leq 10000$, although some $n_{k+1}^*\geq n_k^*+3$ appears when $10000 \leq k \leq 20000$. 
We were not able to detect any case where $n_{k+1}^*= n_k^*+1$. 
\end{remark}
\vspace{2mm}

The proofs of our results shall be carried out in Section\link\ref{proofs}. An important novelty in our method is the use of a family of general higher order Hardy inequalities which enable us to provide a unified treatment that includes the classical cases $k = 1, 2$ as special instances.

\section{Proofs of the results}\label{proofs}
First, we state a general result which could prove useful also for further studies on the high order extrinsic energies \eqref{2s-energia-ext}, \eqref{2s+1-energia-ext}.
\begin{proposition}\label{Th-E-L-equation-k-energy-extr}
Let $(M,g)$ denote a compact Riemannian manifold. Assume that $k \geq 2$ and let $u \in W^{k,2}\left ( M,\s^n \right )$. Then $u$ is an extrinsic weakly $k$-harmonic map if and only if   
\begin{equation}\label{Euler-Lag-equation-k-energy-ext}
\Delta^k u + \lambda_k u=0 
\end{equation}
in the sense of the distributions. Moreover, if \eqref{Euler-Lag-equation-k-energy-ext} holds, then
\begin{equation}\label{lambda-explicit-general}
\lambda_k= \Delta^{k-1}\left ( |\nabla u|^2\right )+ \sum_{j=0}^{k-2}\Delta^j \left ( \langle \Delta^{k-1-j}u,\Delta u \rangle \right )
+2\sum_{j=0}^{k-2}\Delta^j \left ( \nabla \Delta^{k-1-j}u . \nabla u  \right )
\end{equation}
(note that, here and below, $\Delta^0 u=u$).
\end{proposition}
Proposition\link\ref{Th-E-L-equation-k-energy-extr} is a straightforward extension of Proposition 1.1 of \cite{MR1692148}. Therefore, in order to preserve the natural flow of the exposition, we postpone its proof to the final part of this section.
\begin{proof}[Proof of Proposition\link\ref{prop-eq-map}]
Let $r=|x|$ and $u:B^n \backslash \{ O\} \to \s^n \hookrightarrow \R^{n+1}$ be a map of the following form:
\begin{equation}\label{general-map}
   x= \left( x_1, \ldots,x_n\right) \mapsto 
   \left( p(r) \, x, 0\right) =\left( p(r) \, x_1, \ldots,p(r)\,x_n, 0\right) \,\,,
\end{equation}
where $p(r)$ is any smooth function for $r>0$. Then a simple computation yields:
\begin{equation}\label{formule-lemma-secondo}
\begin{split}
\nabla  u &=\left ( \frac{p'\,x_1}{r}
\begin{bmatrix} x_1 \\ 
x_2 \\
 \vdots \\
 x_n \\
\end{bmatrix}
+
\begin{bmatrix}
p\\
0\\
\vdots\\
0
\end{bmatrix},
\frac{p'\,x_2}{r} 
\begin{bmatrix}
 x_1 \\ 
x_2 \\
 \vdots \\
 x_n \\
\end{bmatrix}
+
\begin{bmatrix}
0\\
p\\
\vdots\\
0
\end{bmatrix}, \ldots, \frac{p'\,x_n}{r}\begin{bmatrix}
x_1 \\ 
x_2 \\
 \vdots \\
 x_n \\
\end{bmatrix}+
\begin{bmatrix}
0\\
0\\
\vdots\\
p
\end{bmatrix}, \,\begin{bmatrix} 
0 \\ 
0 \\
 \vdots \\
 0 \\
\end{bmatrix} \right )\\ 
|\nabla  u |^2&= r^2 \,p'^2 +n\, p^2+2r\,  p\,p' \\
\Delta u &= \Big( \left [p''+\frac{n+1}{r}\,p' \right ] \, x, 0\Big )\\ 
| \Delta u|^2 &=r^2\,  \left [p''+\frac{n+1}{r}\,p' \right ]^2\,. \\ 
\end{split}
\end{equation}
Now, we observe that the equator map $u^*$ is a map of type \eqref{general-map} with $p(r)=1/r $. Then, using \eqref{formule-lemma-secondo} together with a routine induction argument we deduce that 
\begin{equation}\label{Delta-j-equator-explicit}
\left (\Delta^k u^* \right)(x)=\left(\frac{A_{k}(n)}{r^{2k+1}} \,\, x,0\right ) \,,\qquad (k \geq1 )\,,
\end{equation}
where $A_{k}(n)$ is the constant defined in \eqref{def-A-j-n}. Next, using \eqref{formule-lemma-secondo} with
\[
p(r)=\frac{A_{s}(n)}{r^{2s+1}} 
\]
we obtain:
\begin{equation}\label{Delta-nabla-equator}
\begin{split}
|\Delta^s u^*|^2&= A_{s}^2(n)\,\, \frac{1}{r^{4s}} \\
|\nabla\Delta^s u^*|^2&= A_{s}^2(n)\,\,\frac{n+4 s^2 -1}{r^{4s+2}}\,.
\end{split}
\end{equation}
It is easy to deduce from \eqref{Delta-nabla-equator} that
\begin{equation*}\label{equator-belongs-Sobolev}
u^*:B^n \to \s^n \in W^{k,2}\left(B^n,\s^n \right ) \quad {\rm iff} \quad n \geq 2k+1 \,.
\end{equation*}
Moreover, we deduce from \eqref{Delta-j-equator-explicit} that
\begin{equation*}\label{Delta-j-equator-multiple-equator}
\Delta^k u^*=\frac{A_{k}(n)}{r^{2k}} \, u^*\,. 
\end{equation*}
Then $u^*$ satisfies \eqref{Euler-Lag-equation-k-energy-ext} strongly on $B^n$ except at the origin, so it is a weak critical point and the proof of Proposition\link\ref{prop-eq-map} is completed.
\end{proof}

An important ingredient in the proof of Theorem\link\ref{Th-equator-map-k>4} are the following high order stability inequalities for the equator map. To shorten the notation, when the meaning is clear, we shall use ``$.$'' instead of $\langle\,\,,\,\rangle$.
\begin{proposition}\label{Th-stab-inequality-general} \quad
\begin{itemize}
\item[{\rm (i)}] Let $k=2s$. If $u^*:B^n\to \s^n$ is stable for the extrinsic $k$-energy $E_{k}^{{\rm ext}}(u)$, then
\begin{equation}\label{stab-ineq-2k}
 \int_{B^n}  | \Delta^s \phi| ^2 -\left ( \Delta^{2s} u^* . u^* \right ) |\phi|^2 
   \, dx \geq 0   
\end{equation}
for all $\phi \in C^{\infty}_0\left (B^n,\R^{n+1} \right ) $.
\item[{\rm (ii)}] Let $k=2s+1$. If $u^*:B^n\to \s^n$ is stable for the extrinsic $k$-energy $E_{k}^{{\rm ext}}(u)$, then
\begin{equation}\label{stab-ineq-2k+1}
 \int_{B^n}  |  \nabla  \Delta^s \phi| ^2 +\left ( \Delta^{2s+1} u^* . u^* \right ) |\phi|^2    \,dx \geq 0
\end{equation}
for all $\phi \in C^{\infty}_0\left (B^n,\R^{n+1} \right ) $.
\end{itemize}
\end{proposition}
\begin{proof} (i) First, we consider variations of the type
\begin{equation*}\label{def-ut}
u_t={u^*+t \phi_{\eta}\over \sqrt{1+t^2\,\eta^2} }\,,
\end{equation*} 
where $ \phi_{\eta} = (0,..,0, \eta) \in C^{\infty}_0(B^n, \R^{n+1})$, with $\eta \in C^{\infty}_0\left (B^n,\R \right ) $. The variation $u_t$ is smooth except at the origin. So, we deduce that on $B^n \setminus \{O\}$
\begin{equation}  \label{1.derivative}
\left .{d \over dt}( \Delta^s{u_t})\right.\Big|_{t=0}=  \Delta^s\left(\left .{d \over dt}(u_t) \right)\right.\Big|_{t=0}= \Delta^s\left(\left .{d \over dt}(u_t)\right.\Big|_{t=0} \right)=  \Delta^s \phi_{\eta}
 \end{equation}
and 
\begin{equation} \label{2.derivative}
\left .{d^2 \over dt^2}( \Delta^s{u_t})\right.\Big|_{t=0}=  \Delta^s\left(\left .{d^2 \over dt^2}(u_t) \right)\right.\Big|_{t=0}=  \Delta^s\left(\left .{d^2 \over dt^2}(u_t)\right.\Big|_{t=0} \right)=  -\,  \Delta^{s} (u^* \,\eta^2) \,.
 \end{equation}
Now, $u^*$ is stable with respect to the variations $u_t$ if 
\begin{equation*} 
\left . {d^2 \over dt^2}\Big ( E_{k}^{{\rm ext}}(u_t)\Big) \right.\Big|_{t=0} \geq 0 \,.
\end{equation*}
But 
\begin{equation}  \label{2.variation} 
\left . {d^2 \over dt^2}\Big ( E_{k}^{{\rm ext}}(u_t)\Big) \right. \Big|_{t=0} = 2 \int_{B^n} \left . {d^2 \over dt^2} \left(\Delta^s(u_t)\right)\right.\Big|_{t=0} . \left .\left(\Delta^s u_t\right )\right. \Big|_{t=0}  \, dx \, +   2 \int_{B^n}  \left\vert {d \over dt} \left .\left(\Delta^s(u_t)\right)\right. \Big|_{t=0} \right\vert^2 \, dx \,.
\end{equation}
Substituting \eqref{1.derivative} and \eqref{2.derivative} into \eqref{2.variation} we get 
\begin{equation*}  \int_{B^n} \big [ | \Delta^s \eta \vert^2 - \Delta^s u^* . \Delta ^{s} (u^* \, \eta^2) 
   \big ]\, dx \geq 0   
\end{equation*}
for any arbitrary smooth function  $\eta  \in C^{\infty}_0\left (B^n,\R \right ) $. 
Next, let  $\phi=(\phi^1,\ldots,\phi^{n+1}) \in C^{\infty}_0\left (B^n,\R^{n+1} \right ) $ and choose $\eta  = \phi^i$ in the above inequality. Summing over $i$ from $1 $ to $n + 1$ , if $u^* $ is stable it follows that
\begin{equation}\label{88}
 \int_{B^n} \big [ | \Delta^s \phi| ^2 - \Delta^s u^* . \Delta ^{s} (u^*  |\phi|^2) 
   \big ]\, dx \geq 0   
\end{equation}
for all $\phi \in C^{\infty} _0\left (B^n,\R^{n+1} \right ) $. Now \eqref{stab-ineq-2k} can be deduced easily from \eqref{88} using the Green identity. The proof of the inequality \eqref{stab-ineq-2k+1} is analogous and so we omit it.
\end{proof}
\vspace{2mm}
In the following lemma we show that an extended version of the stability inequality implies the energy minimizing property. More precisely, we have:
\begin{lemma}\label{prop-stable-imply-minimising}\quad
\begin{itemize}
\item[{\rm (i)}] Let $k=2s$ and $n\geq 2k+1$. If
\begin{equation}\label{stab-ineq-2k-bis}
 \int_{B^n}  | \Delta^s \phi| ^2 - \left (\Delta^{2s} u^* . u^*  \right )\,|\phi|^2
   \, dx \geq 0   
\end{equation}
for all $\phi \in W^{k,2}_0\left (B^n,\R^{n+1} \right ) $, then $u^*$ is energy minimizing for the extrinsic $k$-energy.
\item[{\rm (ii)}] Let $k=2s+1$ and $n\geq 2k+1$. If
\begin{equation}\label{stab-ineq-2k+1-bis}
 \int_{B^n}  |  \nabla  \Delta^s \phi| ^2 +\left ( \Delta^{2s+1} u^* . u^* \right )\,|\phi|^2   \,dx \geq 0
\end{equation}
for all $\phi \in W^{k,2}_0\left (B^n,\R^{n+1} \right ) $, then $u^*$ is energy minimizing for the extrinsic $k$-energy.
\end{itemize}
\end{lemma}
\begin{proof} (i) We must show that
\begin{equation}\label{eq-minimising-2k}  
E_{k}^{{\rm ext}}(u^*) \leq E_{k}^{{\rm ext}} (v) 
\end{equation}
for all $v \in W^{k,2}(B^n,\s^n)$ such that $u^*-v \in  W^{k,2}_0(B^n,\R^{n+1})$. 

On $B^n \setminus \{O\}$ the equator map $u^*$ satisfies 
\begin{equation}\label{u-*-equaz} 
\Delta^{2s}u^* = \left ( \Delta^{2s}u^* .u^* \right ) u^* 
\end {equation}
strongly. Thus we can multiply both sides of \eqref{u-*-equaz} by $\phi \in W^{k,2}_0(B^n,\R^{n+1})$ and we obtain:
\begin{equation*} 
\int_{B^n} \Delta^s u^*  . \Delta^s \phi \ dx = \int_{B^n}\left  ( \Delta^{2s}u^* .u^* \right ) u^* . \phi\, dx 
\end{equation*} 
Choosing $\phi = u^* -v$  we have
\[
\int_{B^n} \Delta^s u^*  . \Delta^s u^* \ dx -\int_{B^n} \Delta^s u^*  . \Delta^s v \ dx =\int_{B^n}\left  ( \Delta^{2s}u^* .u^* \right )\, dx-\int_{B^n}\left  ( \Delta^{2s}u^* .u^* \right ) u^* . v\, dx   
\]
from which we deduce
\begin{equation} \label{eq-min-1}
\int_{B^n} \Delta^s u^* \ . \Delta^s v \ dx = \int_{B^n}\left  ( \Delta^{2s}u^* .u^* \right ) u^* . v\, dx   \,.
\end{equation} 
Next, we apply the hypothesis \eqref{stab-ineq-2k-bis} with $\phi = u^* -v$. Since $u^*,v$  have values in $\s^n$ we have
\[
\left |u^* -v \right |^2=2-2\,u^*.v
\]
and so we easily find
\begin{equation} \label{eq-min-2}
\int_{B^n} \mid \Delta^s v \mid^2  \ dx - \int_{B^n} \mid \Delta^s u^* \mid^2 \ dx   - 2 \int_{B^n} \Delta^s u^* \ . \Delta^s v \ dx \ + 2  \int_{B^n}\left  ( \Delta^{2s}u^* .u^* \right ) u^* . v\, dx   \geq 0 \,.
\end{equation} 
Combining \eqref{eq-min-1} and \eqref{eq-min-2} we obtain precisely \eqref{eq-minimising-2k}. The proof of part (ii) is analogous and so we omit the details.
\end{proof}
Next, we recall the following well-known higher order Hardy inequalities (see, for instance,  \cite{MR2215561}):
\begin{theorem}[\textbf{Hardy Inequalities}]\label{Th-Hardy-ineq} Let $\alpha_k(n)$ be the constants defined in \eqref{def-alpha-k}. 
\begin{itemize}
\item[(i)] Let $k=2s$ and $n\geq 2k+1$. Then
\begin{equation}\label{Hardy-ineq-2k}
 \int_{B^n}  | \Delta^s \phi| ^2\,dx \geq \alpha_{2s}(n)\int_{B^n}  \frac{|\phi| ^2}{|x|^{4s}}\,dx \quad\quad  (s \geq 1)\,.
 \end{equation}
for all $\phi \in W^{k,2}_0(B^n,\R^{n+1})$.
 \item[(ii)] Let $k=2s+1$ and $n\geq 2k+1$. Then
 \begin{equation}\label{Hardy-ineq-2k+1}
 \int_{B^n}  | \nabla\Delta^s \phi| ^2\,dx \geq\alpha_{2s+1}(n)\int_{B^n}  \frac{|\phi| ^2}{|x|^{4s+2}}\,dx \quad (s \geq 0)\,.
  \end{equation}
for all $\phi \in W^{k,2}_0(B^n,\R^{n+1})$.
\end{itemize}
Moreover, the constants $\alpha_k(n)$ in \eqref{Hardy-ineq-2k} and \eqref{Hardy-ineq-2k+1} are the best possible, i.e., by density of $ C^{\infty}_0(B^n,\R^{n+1})$ in $ W^{k,2}_0(B^n,\R^{n+1})$,
\begin{equation}\label{Hardy-best-constant}
\begin{split}
 \alpha_{2s}(n)&= {\rm Inf}  \,\, \left \{ \frac{\displaystyle{\int_{B^n}  | \Delta^s \phi| ^2\,dx} }{\displaystyle{\int_{B^n}  \frac{|\phi| ^2}{|x|^{4s}}\,dx} }\,\, \colon \phi \in C^{\infty}_0(B^n,\R^{n+1}) \right \}\\
 \alpha_{2s+1}(n)&= {\rm Inf}  \,\, \left \{ \frac{\displaystyle{\int_{B^n}  | \nabla\Delta^s \phi| ^2\,dx }}{\displaystyle{\int_{B^n}  \frac{|\phi| ^2}{|x|^{4s+2}}\,dx }} \,\, \colon \phi \in C^{\infty}_0(B^n,\R^{n+1})\right \} \,.
\end{split}
 \end{equation}
\end{theorem}
Finally, we are now in the right position to prove our main result.
\begin{proof}[\textbf{Proof of Theorem~\ref{Th-equator-map-k>4}}] First, we prove the statement (i). To this purpose, according to Lemma\link\ref{prop-stable-imply-minimising}, it suffices to show that $u^*$ verifies the stability inequalities \eqref{stab-ineq-2k-bis}, \eqref{stab-ineq-2k+1-bis}.

Using \eqref{formule-lemma-secondo} and \eqref{Delta-j-equator-explicit}, we compute the relevant terms on $B^n\setminus\{O\}$:
\begin{equation}\label{terms-stab-in-u*-k}
\begin{split}
\Delta^{2s}u^* . u^*&=\frac{A_{2s}(n)}{r^{4s}} \\
\Delta^{2s+1}u^* . u^*&=\frac{A_{2s+1}(n)}{r^{4s+2}} \,.
\end{split}
\end{equation}

Now, using \eqref{Hardy-ineq-2k} and \eqref{terms-stab-in-u*-k} we obtain:
\begin{eqnarray*}
\int_{B^n}  | \Delta^s \phi| ^2 -  \left [ \Delta^{2s}u^* .u^*\right] |\phi|^2\, dx 
 &\geq& \int_{B^n} \Big ( \alpha_{2s}(n)- A_{2s}(n) \Big )\, \frac{|\phi|^2}{r^{4s}}\, dx\\
 &=& \int_{B^n}  \, \mathcal{P}_{2s}(n)\, \, \frac{|\phi|^2}{r^{4s}}\, dx \,.
\end{eqnarray*}
Now, since $\mathcal{P}_{2s}(n)\geq 0$ by assumption, the stability inequality \eqref{stab-ineq-2k-bis} holds and then we can apply Lemma\link\ref{prop-stable-imply-minimising} to conclude that $u^*$ is minimizing, so ending the proof of (i) in the case that $k=2s$ is even. Similarly, in the case of the odd extrinsic energies we have:
\begin{eqnarray*}
\int_{B^n}  | \nabla \Delta^s \phi| ^2 + \left [ \Delta^{2s+1}u^* .u^*\right] |\phi|^2\, dx 
 &\geq& \int_{B^n} \Big ( \alpha_{2s+1}(n)+ A_{2s+1}(n) \Big )\, \frac{|\phi|^2}{r^{4s+2}}\, dx\\
 &=& \int_{B^n}  \, \mathcal{P}_{2s+1}(n)\, \, \frac{|\phi|^2}{r^{4s+2}}\, dx \, \geq 0
\end{eqnarray*}
and so the conclusion of (i) is immediate.

(ii) Now, we have to prove that $u^*:B^n \to \s^n$ is unstable for the extrinsic $k$-energy when $\mathcal{P}_k(n)<0$. We start with the case that $k=2s$ and choose a small $\varepsilon>0$ such that 
\begin{equation}\label{epsilon>0}
\mathcal{P}_{2s}(n)+\varepsilon <0 \,.
\end{equation}
Next, we deduce from \eqref{Hardy-best-constant} that there exists a not identically zero $\phi_\varepsilon \in C^{\infty}_0(B^n,\R^{n+1})$ such that
\begin{equation}\label{phi-epsilon}
  \frac{\displaystyle{\int_{B^n}  | \Delta^s \phi_\varepsilon| ^2\,dx }}{\displaystyle{\int_{B^n}  \frac{|\phi_\varepsilon| ^2}{|x|^{4s}}\,dx }} \leq \alpha_{2s}(n)+\varepsilon\,.
\end{equation}

We apply \eqref{terms-stab-in-u*-k} and \eqref{phi-epsilon} in the left side of the stability inequality \eqref{stab-ineq-2k}:
\begin{eqnarray*}
\int_{B^n}  | \Delta^s \phi_\varepsilon| ^2 -  \left [ \Delta^{2s}u^* .u^*\right] |\phi_\varepsilon|^2\, dx 
 &\leq& \int_{B^n} \Big ( \alpha_{2s}(n)+\varepsilon- A_{2s}(n) \Big )\, \frac{|\phi_\varepsilon|^2}{r^{4s}}\, dx\\
 &=& \int_{B^n}  \, \Big (\mathcal{P}_{2s}(n)+\varepsilon\Big ) \, \frac{|\phi_\varepsilon|^2}{r^{4s}}\, dx \,<\,0
\end{eqnarray*}
by \eqref{epsilon>0}. The case $k=2s+1$ is analogous and so the proof of Theorem\link\ref{Th-equator-map-k>4} is ended.
\end{proof}
\begin{remark}\label{Rem-phi-epsilon}
We point out that, in the case of the extrinsic $2$-energy, a function $\phi_\varepsilon\in W^{2,2}_0(B^n,\R^{n+1})$ as in \eqref{phi-epsilon} was constructed explicitly by Hong and Thompson in their proof of Theorem\link\ref{Th-equator-map-Hong} (see \cite{MR2322746}).
\end{remark}

\begin{proof}[Proof of Lemma\link\ref{lemma-Pj(n)}] We give the details in the case that $k=2s$. In order to prove \eqref{Pj(2j+1)<0} we observe that, carefully rearranging the various factors, we have:
\begin{equation}\label{eq-lemma-tecnico-1}
 \prod_{i=1}^s (4s+1-4i)^2 (4s+4i-3)^2 =\prod_{i=0}^{2s-1}(1+4i)^2 \,.
 \end{equation} 
 Indeed,
 \begin{eqnarray*}
 \prod_{i=1}^s (4s+1-4i)^2 (4s+4i-3)^2&=&\prod_{i=1}^s (4s+1-4i)^2 \,\,\prod_{i=1}^s(4s+4i-3)^2 \\
 &=&\prod_{i=0}^{s-1}(1+4i)^2 \,\,\prod_{i=s}^{2s-1}(1+4i)^2=\prod_{i=0}^{2s-1}(1+4i)^2 \,.
 \end{eqnarray*}
Now we observe that, according to \eqref{def-alpha-k}, 
\[
\alpha_{2s}(4s+1)= \frac{1}{16^s}\,\prod_{i=1}^s (4s+1-4i)^2 (4s+4i-3)^2 \,.
\]
Thus, using \eqref{eq-lemma-tecnico-1}, we deduce that
\begin{equation}\label{eq-lemma-tecnico-2}
 \alpha_{2s}(4s+1) = \frac{1}{16^s}\,\prod_{i=0}^{2s-1}(1+4i)^2 =\prod_{i=0}^{2s-1}\left (4i^2+2i+ \frac{1}{4}\right )\,.
 \end{equation}
Next, we start from \eqref{def-A-j-n} and proceed in a similar fashion:
\begin{eqnarray}\label{eq-lemma-tecnico-3}\nonumber
 A_{2s}(4s+1)&=&\prod_{i=1}^{2s} (4s+2-2i)(2i-1) \\
 &=&\prod_{i=0}^{2s-1}(4s-2i) \,\,\prod_{i=0}^{2s-1}(2i+1)\\\nonumber
 &=&\prod_{i=0}^{2s-1}(2+2i) \,\,\prod_{i=0}^{2s-1}(2i+1)=\prod_{i=0}^{2s-1}(4i^2+6i+2) \,.
 \end{eqnarray}
Finally, comparing \eqref{eq-lemma-tecnico-2} and \eqref{eq-lemma-tecnico-3}, we conclude immediately that
 \[
 \mathcal{P}_{2s}(4s+1)=\alpha_{2s}(4s+1)- A_{2s}(4s+1)< 0 
 \]
 for all $ s \geq1$. 
 
Next, we prove the existence of $n_k^*$. First, we observe that the highest order coefficient of $\mathcal{P}_{2s}(n)$ is positive and so $\lim_{n \to +\infty}\mathcal{P}_{2s}(n)=+\infty$. Therefore, because of \eqref{Pj(2j+1)<0}, we deduce that there exists an integer $N> 4s+1$ such that $\mathcal{P}_{2s}(N)>0$. It follows, by an obvious induction argument, that the proof is completed if we show that the following statement is true:
\begin{equation}\label{eq-induction}
{\rm If} \,\,N> 4s+1\,\,{\rm and} \,\,\mathcal{P}_{2s}(N)\geq 0\,, \,{\rm then} \,\,\mathcal{P}_{2s}(N+1)> 0\,.
\end{equation}
To this purpose, in a similar fashion to \eqref{eq-lemma-tecnico-1}, we write
\begin{eqnarray}\label{eq-proof-eq-induction-1}\nonumber
\alpha_{2s}(N)&=& 16^{-s}\,\,  \prod_{i=0}^{2s-1}(N-4s+4i)^2\\\nonumber
\alpha_{2s}(N+1)&=& 16^{-s}\,\, \prod_{i=0}^{2s-1}(N+1-4s+4i)^2\\
&=& 16^{-s}\,\, \prod_{i=0}^{2s-1}(N-4s+4i)^2 \,\gamma_i \\\nonumber
&=& \alpha_{2s}(N)\,\, \prod_{i=0}^{2s-1} \gamma_i \,,
\end{eqnarray}
where we have set
\begin{equation*}\label{def-gamma-i}
\gamma_i=\frac{(N+1-4s+4i)^2}{(N-4s+4i)^2}\,.
\end{equation*}
Similarly, we have:
\begin{eqnarray}\label{eq-proof-eq-induction-2}\nonumber
A_{2s}(N)&=&\prod_{i=0}^{2s-1}\, (2i+1)\,\, \prod_{i=0}^{2s-1}(N-1-2i)=\prod_{i=0}^{2s-1}\, (2i+1)\,\,  \prod_{i=0}^{2s-1}(N-4s+1+2i)\\\nonumber
A_{2s}(N+1)&=&\prod_{i=0}^{2s-1}\, (2i+1)\,\,  \prod_{i=0}^{2s-1}(N-4s+2+2i)\\\nonumber
&=& \prod_{i=0}^{2s-1}\, (2i+1)\,\,  \prod_{i=0}^{2s-1}(N-4s+1+2i) \,\beta_i \\
&=& A_{2s}(N)\, \, \prod_{i=0}^{2s-1} \,\beta_i \,,
\end{eqnarray}
where we have set
\begin{equation*}\label{def-beta-i}
\beta_i=\frac{(N-4s+2+2i)}{(N-4s+1+2i)}\,.
\end{equation*}
We point out that, since $N>4s+1$, all the numerators and denominators in $\gamma_i, \beta_i$ are positive. Moreover, it is easy to check that 
\[
\gamma_i > \beta_i \qquad \quad  (0 \leq i \leq 2s-1)\,.
\]
In fact, since $N>4s+1$ putting $N=4s+1+a$ with $a>0$, the quantity $\gamma_i - \beta_i$ becomes
\[
\frac{a^2+(5+4 i) a+ 14 i +5}{(a + 2 i + 2) (a + 4 i + 1)^2}
\]
which is positive for any $i\geq 0$ and any $a>0$.
Thus
\begin{equation}\label{eq-proof-eq-induction-3}
\prod_{i=0}^{2s-1} \gamma_i \,-\, \prod_{i=0}^{2s-1} \beta_i >0 \,.
\end{equation}
Now, using \eqref{eq-proof-eq-induction-1}, \eqref{eq-proof-eq-induction-2} and \eqref{eq-proof-eq-induction-3} it is easy to end the proof of \eqref{eq-induction}. Indeed, since $\alpha_{2s}(N)$ and $A_{2s}(N)$ are positive,
\begin{eqnarray*}
\mathcal{P}_{2s}(N+1)&=&\alpha_{2s}(N+1)-A_{2s}(N+1)\\
&=&  \alpha_{2s}(N)\, \prod_{i=0}^{2s-1} \gamma_i-A_{2s}(N)\, \prod_{i=0}^{2s-1}\beta_i\\
&=& \Big ( \alpha_{2s}(N)-A_{2s}(N)\Big )\, \prod_{i=0}^{2s-1} \gamma_i \,+A_{2s}(N)\,\Big ( \prod_{i=0}^{2s-1} \gamma_i-\prod_{i=0}^{2s-1}\beta_i \Big )\\
&=& \mathcal{P}_{2s}(N)\,\prod_{i=0}^{2s-1} \gamma_i \,+A_{2s}(N)\,\Big ( \prod_{i=0}^{2s-1} \gamma_i-\prod_{i=0}^{2s-1}\beta_i \Big)\\
&>&0\,.
\end{eqnarray*}
The case that $k=2s+1$ can be handled with the same method and so we omit further details.
\end{proof}
\begin{proof}[Proof of Proposition\link\ref{Th-E-L-equation-k-energy-extr}] If we considered $u$ as a map into $\R^{n+1}$ the Euler-Lagrange equation of the functional $E_{k}^{{\rm ext}}(u)$ would simply be $\Delta^k u=0$. Therefore, when we regard $u$ as a map into $\s^n$, we conclude that $u$ is a critical point of $E_{k}^{{\rm ext}}(u)$ provided that $\Delta^k u \perp T_u\s^n$ in the sense of distributions (see also \cite[Proposition~2.1]{MR2054520}), i.e., if \eqref{Euler-Lag-equation-k-energy-ext} holds. So, it only remains to determine the Lagrange multiplier $\lambda_k$. In the case $k=2$, the computation of $\lambda_2$ was achieved in \cite[proof of Proposition 1.1]{MR1692148} using the Green identity and integration by parts in a single step. Here we have to repeat the same process $k-1$ times. Let us describe this in detail. 

\textbf{Step 1:}

First, we inner product both sides of equation \eqref{Euler-Lag-equation-k-energy-ext} with $u$, we multiply by a compactly supported test function $\varphi \in C^{\infty}_0(M)$ and integrate over $M$. To simplify notation, we omit to write the volume element $dv_g$. Using the Green identity and integrating by parts we have
\begin{eqnarray*}\label{computation-lambdak}\nonumber
 - \int_M \lambda_k\, \varphi &=& \int_M \langle \Delta^k u, u \rangle \varphi=\int_M \langle \Delta^{k-1} u, \Delta(u \varphi)\rangle \\\nonumber
 &=&\int_M \langle \Delta^{k-1}u, \Delta u \rangle \varphi+ 2\langle \Delta^{k-1} u, \nabla u . \nabla \varphi \rangle+\langle \Delta^{k-1} u, u \rangle \Delta \varphi\\
  &=&\int_M \langle \Delta^{k-1} u, \Delta u \rangle \varphi- 2\langle  \Delta^{k-1} u, \Delta u  \rangle \varphi- 2\left ( \nabla \Delta^{k-1} u . \nabla u  \right ) \varphi+\langle \Delta^{k-1} u, u \rangle \Delta \varphi\\\nonumber
  &=&\int_M -\Big [ \langle  \Delta^{k-1} u, \Delta u  \rangle + 2\left ( \nabla \Delta^{k-1} u . \nabla u  \right ) \Big ] \varphi+\int_M \langle \Delta^{k-1} u, u \rangle \Delta \varphi\nonumber
\end{eqnarray*}
(note that we write $\nabla u . \nabla \varphi$ to denote the obvious version of \eqref{def-scalar-prod} which yields as a result the $(n+1)$ vector whose $j$-entry is $\langle \nabla u^j, \nabla \varphi \rangle$).

Now, the two terms within the brackets $[\,\,]$ have the correct form. Indeed, they are the terms corresponding to $j=0$ in the two sums which define $\lambda_k$ in \eqref{lambda-explicit-general}. Then we can start Step 2 to deal with the remaining term $\int_M \langle \Delta^{k-1} u, u \rangle \Delta \varphi$.

\textbf{Step 2:}
\begin{eqnarray*}\label{computation-lambdak-bis}\nonumber
\int_M \langle \Delta^{k-1} u, u \rangle \Delta \varphi &=& \int_M \langle \Delta^{k-2} u, \Delta(u \Delta \varphi)\rangle \\\nonumber
 &=&\int_M \langle \Delta^{k-2}u, \Delta u \rangle \Delta\varphi+ 2\langle \Delta^{k-2} u, \nabla u . \nabla \Delta\varphi \rangle+\langle \Delta^{k-2} u, u \rangle \Delta^2 \varphi\\\nonumber
  &=&\int_M \langle \Delta^{k-2} u, \Delta u \rangle \Delta\varphi- 2\langle  \Delta^{k-2} u, \Delta u  \rangle \Delta\varphi- 2\left ( \nabla \Delta^{k-2} u . \nabla u  \right )\Delta \varphi+\langle \Delta^{k-2} u, u \rangle \Delta^2 \varphi\\
  &=&\int_M -\Big [ \langle  \Delta^{k-2} u, \Delta u  \rangle + 2\left ( \nabla \Delta^{k-2} u . \nabla u  \right ) \Big ] \Delta\varphi+\int_M \langle \Delta^{k-2} u, u \rangle \Delta^2 \varphi\\\nonumber
  &=&\int_M -\Big [ \Delta\langle  \Delta^{k-2} u, \Delta u  \rangle + 2\Delta\left ( \nabla \Delta^{k-2} u . \nabla u  \right ) \Big ] \varphi+\int_M \langle \Delta^{k-2} u, u \rangle \Delta^2 \varphi\nonumber\,.
\end{eqnarray*}
Now, the two terms within the brackets $[\,\,]$ have the correct form. Indeed, they correspond to $j=1$ in the two sums which define $\lambda_k$ in \eqref{lambda-explicit-general}. Of course, now the remaining term $\int_M \langle \Delta^{k-2} u, u \rangle \Delta^2 \varphi$ shall be treated analogously in Step 3. 

By way of summary, after $k-2$ steps of this type we recover all the terms with $j=0,\ldots, k-3$ associated to the two sums in \eqref{lambda-explicit-general}, and we are left with the term $\int_M \langle \Delta^2 u, u \rangle \Delta^{k-2} \varphi$.

\textbf{Step $k-1$:}
\begin{eqnarray*}
 \int_M \langle \Delta^2u, u \rangle \Delta^{k-2} \varphi&=& \int_M \langle \Delta u, \Delta(u \Delta^{k-2}\varphi)\rangle \\
 &=&\int_M \langle \Delta u, \Delta u \rangle \Delta^{k-2}\varphi+ 2\langle \Delta u, \nabla u . \nabla \Delta^{k-2}\varphi \rangle+\langle \Delta u, u \rangle \Delta^{k-1} \varphi\\
  &=&\int_M \big [ -\langle  \Delta u, \Delta u  \rangle - 2\left ( \nabla \Delta u . \nabla u  \right ) \big ] \Delta^{k-2} \varphi-|\nabla u|^2 \Delta^{k-1} \varphi\\
  &=&\int_M -\Big [ \Delta^{k-2} \left ( \langle  \Delta u, \Delta u  \rangle\right ) + 2 \Delta^{k-2} \left (  \nabla \Delta u . \nabla u  \right ) + \Delta^{k-1} \left ( |\nabla u|^2 \right ) \Big ] \varphi\,,
\end{eqnarray*}
where, for the third equality, we have used $\langle \Delta u, u \rangle=-|\nabla u|^2$. Now, the first two terms correspond to $j=k-2$ in the two sums in \eqref{lambda-explicit-general}, while the third term coincides with the first one in \eqref{lambda-explicit-general}. Therefore, the proof of Proposition\link\ref{Th-E-L-equation-k-energy-extr} is completed.
\end{proof}

\section{Appendix}\label{Appendix}

In the following \textit{Wolfram Mathematica}${}^{ \text{\textregistered}} $ code  {\tt P2[s,n]} and  {\tt a2[s,n]} represent $\mathcal{P}_{2s}(n)$ and $\alpha_{2s}(n)$ respectively. 
Evaluating the following cell\vspace{3mm}

\begin{lstlisting}[language=Mathematica,caption={}]
a2[s_,n_] :=  Product[((1/4) (n - 4 j) (n + 4 j - 4))^2, {j, 1, s}];
A[k_,n_]:=(-1)^k Product[(n - 2 i + 1) (2 i - 1), {i, 1, k}];
P2[s_,n_]:=a2[s,n] - A[2s,n];
Do[{i := 4 s + 1; 
Do[If[P2[s, n] < 0, i++, i] , {n, (4 s + 1), 4 (2 s + 1)}], Print["n*"[2 s], "=" [i] ]}, {s, 1000, 1003}]
\end{lstlisting}
\vspace{3mm}

yields the values of $n_{2s}^*$ for $1000\leq s \leq 1003$ and the output is:\vspace{3mm}

\begin{lstlisting}[language=Mathematica,caption={}]
n*[2000]= [4019] 
n*[2002]= [4023] 
n*[2004]= [4027] 
n*[2006]= [4031] 
  \end{lstlisting}
A similar code computes $n_{2s+1}^*$.

\end{document}